\documentclass[11pt, reqno]{amsart}
\usepackage[cp1251]{inputenc}
\usepackage{amssymb,upref}
\usepackage{a4wide}
\usepackage{verbatim}
\usepackage{enumitem}
\usepackage{diagbox} 
\usepackage{graphicx}
\graphicspath{{pictures/}}
\DeclareGraphicsExtensions{.pdf,.png,.jpg}
\usepackage{chngcntr}
\usepackage{amsthm}
\usepackage{subcaption}
\usepackage{booktabs}
\usepackage{float}
\theoremstyle{plain}

\newtheorem{theorem}{Theorem}[section]
\newtheorem{lemma}[theorem]{Lemma}
\newtheorem{corollary}[theorem]{Corollary}

\newtheorem{proposition}[theorem]{Proposition}

\theoremstyle{definition}

\newtheorem{definition}[theorem]{Definition}

\theoremstyle{remark}

\def\N{\ensuremath{\mathbb N}}

\usepackage{xcolor}

\begin{document}

\title{From Dimension Drop to Aperiodic Order}

\author[N. Jurga]{Natalia Jurga}\address{Natalia Jurga\\ Mathematical Institute\\
University of St Andrews\\
North Haugh\\
St Andrews\\
KY16 9SS\\
Scotland \\}
\email{naj1@st-andrews.ac.uk}
\urladdr{https://www.nataliajurga.co.uk/}
\author[D. Karvatskyi]{Dmytro Karvatskyi}\address{Dmytro Karvatskyi\\  Institute of Mathematics of NAS of Ukraine\\
3, Tereshchenkivska Street\\
Kyiv \\  01030\\Ukraine}
\email{karvatsky@imath.kiev.ua}
\urladdr{https://dmytro-karvatskyi.github.io/}
\thanks{NJ was partially supported by a Leverhulme Early Career Fellowship (ECF-2021-385). DK worked at the University of St Andrews
within the framework of the Isaac Newton Institute Solidarity Programme and also received
additional support from the London Mathematical Society. The authors thank Thomas Jordan for highlighting the connection to projections of the four-corner Cantor set, and Henna Koivusalo for conversations about aperiodic tilings.}

\begin{abstract}
We consider the parametrised family of sets
\[
E(x,y)=\left\{\sum_{n=1}^{\infty}\frac{\varepsilon_n}{4^n}:
(\varepsilon_n) \in \{0,x,y,x+y\}^{\mathbb{N}}\right\} \;\;\ \textnormal{for $(x,y) \in \N^2$.}
\]

This family can be viewed through three lenses: (a) as homogeneous self-similar sets; (b) as achievement sets of bi-geometric series; or (c) as the set of `rational' orthogonal projections of the four-corner Cantor set. 
We synthesise these three perspectives to obtain a complete topological classification of \(E(x,y)\) for $(x,y) \in \N^2$.

Next, we collapse this topological classification to a binary one according to whether or not $E(x,y)$ has interior. When this binary classification is visualised, it reveals a two-colour tiling $T$ of the lattice \(\mathbb N^2\), which, despite being visibly structured, turns out to be \emph{aperiodic}; indeed, we prove it has no non-trivial translational symmetries. Due to the rigidity of our model, this same binary classification simultaneously captures several dichotomies. Most notably, when the family $\{E(x,y)\}_{(x,y) \in \N^2}$ is viewed through the theory of self-similar sets, $T$ can be seen to describe the emergence of dimension drop within the family.

Finally we examine the mechanism underlying the tiling's aperiodic order. By considering the number-theoretic properties of the tiling, we characterise its substitution structure, and discover that $T$ is a factor of a substitution tiling on  four ``hidden'' arithmetically defined states.

\end{abstract}

\maketitle

\section{Introduction}

Many fundamental problems in fractal geometry share a common objective: to distinguish those fractal objects whose geometric complexity is fully realised from those for which some of this complexity is compressed by arithmetic, algebraic, or geometric structure. For example, the exact overlaps conjecture \cite{Hochman} concerns the distinction between self-similar sets whose Hausdorff dimension attains the expected similarity dimension and those for which overlaps force a dimension drop. Similarly, refinements of Marstrand's projection theorem \cite{ffj} seek to identify the ``exceptional'' projections for which the complexity of a set is unusually compressed under projection. Most of the literature on these topics concerns itself with establishing the conditions needed to rule out phenomena such as dimension drop.

In this paper we approach dimension drop from a different perspective. Rather than treating dimension drop as an exceptional phenomenon in an individual fractal, we investigate the geometric structure that the collection of dimension-drop parameters possesses. To this end, we consider a natural parametrised family of self-similar sets where the parameters at which dimension drop occurs can be explicitly determined. In our model, this set of dimension-drop parameters itself exhibits a rich form of aperiodic order.

The family of sets we consider is defined as

\[
E(x,y)=\left\{\sum_{n=1}^{\infty}\frac{\varepsilon_n}{4^n}:
(\varepsilon_n) \in \{0,x,y,x+y\}^{\mathbb{N}}\right\} \;\;\ \textnormal{for $(x,y) \in \N^2$.}
\]

This family admits several complementary interpretations. It can be viewed simultaneously as a parametrised family of self-similar sets, as a family of achievement sets of bi-geometric series or as the set of rational projections of the four-corner Cantor set; we give further detail on each of these interpretations in \S \ref{persp}. These different perspectives connect the model to several well-developed areas of fractal geometry and additive combinatorics, while retaining a particularly simple two-dimensional parameter space.

\subsection{Aperiodic order of the dimension-drop parameter tiling}

We investigate the set of parameters $(x,y)\in\mathbb{N}^2$ for which
$E(x,y)$ exhibits dimension drop. Rather than forming an unstructured
collection of exceptional parameters, these parameters themselves
organise into a highly structured, but non-periodic, geometric pattern in the parameter plane.

We begin by introducing the arithmetic sets which determine this pattern. Let \(V\) denote the set of positive integers whose last non-zero digit in their quaternary expansion is either \(1\) or \(3\), and let \(D=\mathbb{N}\setminus V\). Equivalently, \(D\) consists of those positive integers whose last non-zero quaternary digit is \(2\). 

Define the characteristic function
\[
c\colon\mathbb{N}\longrightarrow\{0,1\}
\]
of the set \(V\) by
\[
c(n)=
\begin{cases}
1, & \text{if }n\in V,\\
0, & \text{if }n\in D.
\end{cases}
\]

Listing the elements of \(V\) and \(D\) in increasing order gives the sequences \((v_n)\) and \((d_n)\), which are known as the \textbf{vile numbers} and \textbf{dopey numbers}, respectively. Their initial terms are

$$
(v_n)_{n\geq 1}
=
(1,3,4,5,7,9,11,12,13,15,16,17,19,20,21,23,25,27,28,29,31,33,\ldots)
$$

and

$$
(d_n)_{n\geq 1}
=
(2,6,8,10,14,18,22,24,26,30,32,34,38,40,42,46,50,54,56,58,62,\ldots).
$$

These sequences arise in combinatorial game theory \cite{Fra11,Fraenkel} and are listed in the On-Line Encyclopedia of Integer Sequences as A003159 and A036554, respectively. Their characteristic sequences are \(2\)-automatic and are closely related to the Thue--Morse sequence \cite{AAB95}. 

We use these arithmetic sets to define the parameter tiling, which is visualised in Figure~\ref{tiling}.

\begin{definition}\label{def:tiling}
The \emph{parameter tiling} is the two-colouring
\[
T:\mathbb{N}^2\longrightarrow\{0,1\}
\]
defined by
\[
T(x,y)=
\begin{cases}
1, & \text{if } c(x)\neq c(y),\\
0, & \text{if } c(x)=c(y).
\end{cases}
\]
We refer to the parameters in \(T^{-1}(1)\) as \emph{black} and those
in \(T^{-1}(0)\) as \emph{white}.
\end{definition}

\begin{figure}[h]
\center{\includegraphics[scale=0.4]{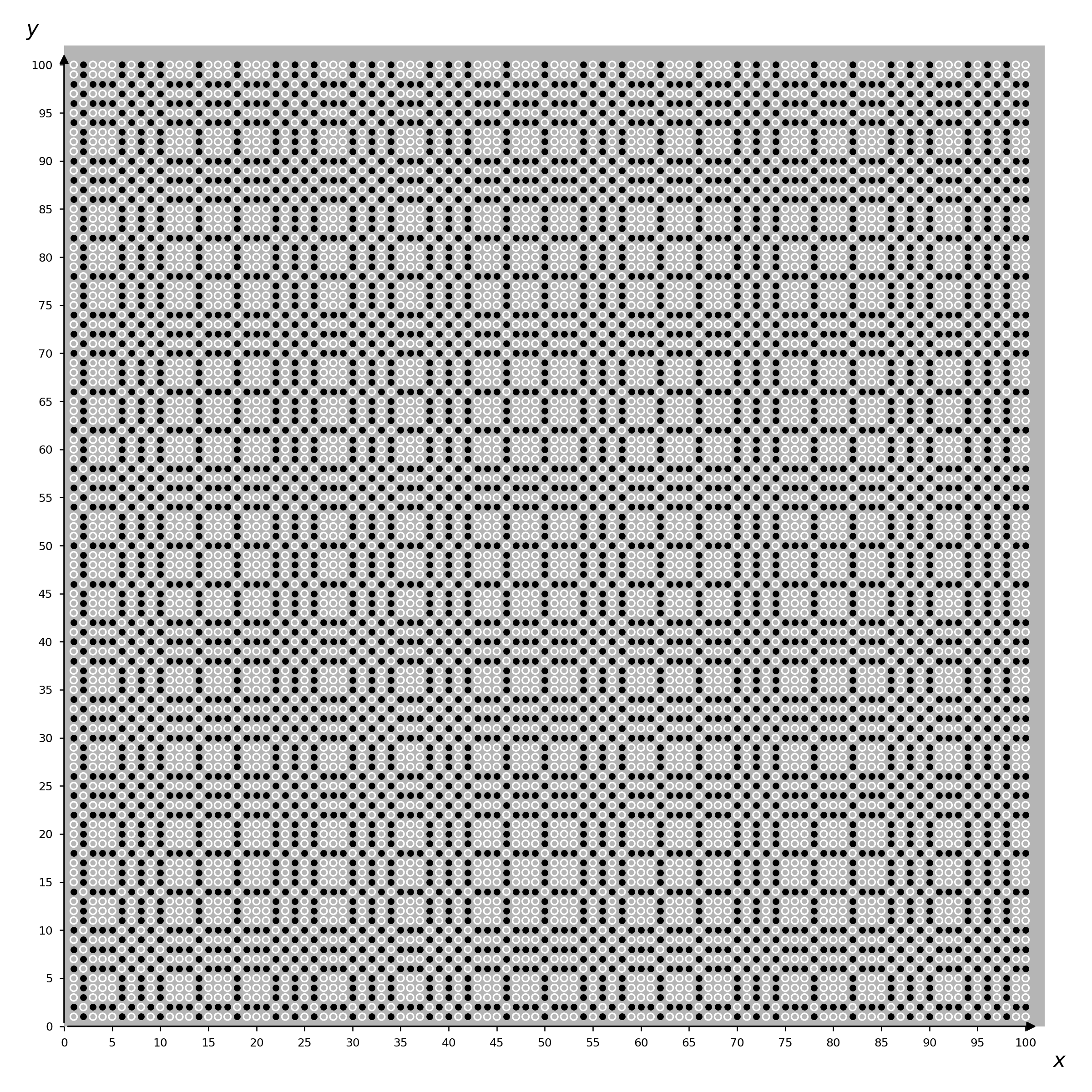}}
\caption{The parameter tiling $T$: white corresponds to \(E(x,y)\) having empty interior (and ``dimension drop''),  black corresponds to \(E(x,y)\) having non-empty interior (and no ``dimension drop'').}
\label{tiling}
\end{figure}

The relevance of the partition \(\mathbb N=V\cup D\) to our fractal family is captured by the following classification.

\begin{theorem}\label{main}
For every \(x,y\in\mathbb{N}\),
\[
T(x,y)=1
\quad\Longleftrightarrow\quad
E(x,y)\text{ has non-empty interior}.
\]
Equivalently,
\begin{enumerate}[label=\arabic*)]
    \item if \(x,y\in V\) or \(x,y\in D\), then \(E(x,y)\) has empty
    interior;
    \item if \(x\in V\) and \(y\in D\), or vice versa, then \(E(x,y)\)
    has non-empty interior.
\end{enumerate}
\end{theorem}

Because of the particularly rigid structure of the family $\{E(x,y)\}_{(x,y) \in \N^2}$, the tiling $T$ simultaneously captures several seemingly different
properties of the set and of the underlying iterated function
system, see Proposition \ref{tableprop}. Namely,

\begin{table}[H]
    \centering
    \caption{Equivalent properties of the black and white parameters.}
    \label{tab}
    \begin{tabular}{cc}
        \toprule
        \textbf{Black} & \textbf{White} \\
        \midrule
        $E(x,y)$ has non-empty interior & $E(x,y)$ has empty interior \\
        $E(x,y)$ has positive measure & $E(x,y)$ has zero measure \\
        $\dim_H E(x,y)=1$ & $\dim_H E(x,y)<1$ \\
        The underlying IFS has no exact overlaps & The underlying IFS has exact overlaps \\
        The underlying IFS satisfies OSC & The underlying IFS does not satisfy OSC \\
        \bottomrule
    \end{tabular}
\end{table}

Perhaps the most striking feature of $T$ is that, despite its apparent order, it is aperiodic. Indeed, in Theorem \ref{aperiodicthm} we prove that $T$ admits no non-trivial translational period. 

We then characterise the mechanism underlying this aperiodic order. In Theorem \ref{subthm} we will show that the tiling admits a constant length substitution structure. In particular, Corollary \ref{cor:factor-substitution} shows that $T$ is a factor of a substitution tiling on four states, determined by the arithmetic type of the pair $(x,y)$ according to the membership of $x$ and $y$ in $V$ and $D$.  

In \S \ref{final}, we place these results in a broader context by discussing how the emergence of aperiodic order in $T$ parallels the appearance of rich structure in classical parameter sets such as the Mandelbrot set and the Rauzy gasket.

\subsection{Full topological classification}

Theorem~\ref{main} gives a binary classification according to whether \(E(x,y)\) has interior. Next, we provide the complete classification of $E(x,y)$ which describes all possible topological types of the attractor.

We say that a set is a Cantor set if it is homeomorphic to the middle-third Cantor set, and we say that a set is a Cantorval if it is homeomorphic to
 $$C_{1/3} \cup \bigcup_{n=1}^{\infty} G_{2n-1} = [0, 1] \setminus \bigcup_{n=1}^{\infty} G_{2n},$$
where $C_{1/3}$ is the middle-third Cantor set and $G_n$ is the union of the $2^{n-1}$ open middle thirds that are removed from $[0, 1]$ at the $n$-th step in the construction of $C_{1/3}$.

\begin{theorem}\label{refinement}
Let \((x,y)\in\mathbb{N}^2\) and consider the self-similar set \(E(x,y)\).

\begin{enumerate}
\item[(a)] If \(x,y\in V\) or \(x,y\in D\), then \(E(x,y)\) is a Cantor set with a non-integer Hausdorff dimension.

\item[(b)] If \(x\in V\) and \(y\in D\), or vice versa, then \(E(x,y)\) is either a finite union of closed intervals or a Cantorval. More precisely:
\begin{enumerate}
\item[(i)] If \(y=2\cdot4^t\cdot x\) for some \(t \in \mathbb{Z}\), then \(E(x,y)\) consists of \(2^\alpha\) closed intervals, where
$$
\alpha=
\left\{
\begin{array}{r}
t, ~ \text{if} ~ ~ t \geq 0,\\
-t-1, ~ \text{if} ~ ~ t < 0.\\

\end{array}
\right.
$$

\item[(ii)] Otherwise, \(E(x,y)\) is a Cantorval.

\end{enumerate}
\end{enumerate}
\end{theorem}

In particular, the two-dimensional parameter space admits a more detailed geometric decomposition according to the precise topological type of \(E(x,y)\). The parameters for which the attractor is a finite union of intervals form the red ``rays'' visible in Figure~\ref{tilingdet}, while the remaining interior-producing parameters correspond to Cantorvals.

\begin{figure}[H]
\center{\includegraphics[scale=0.4]{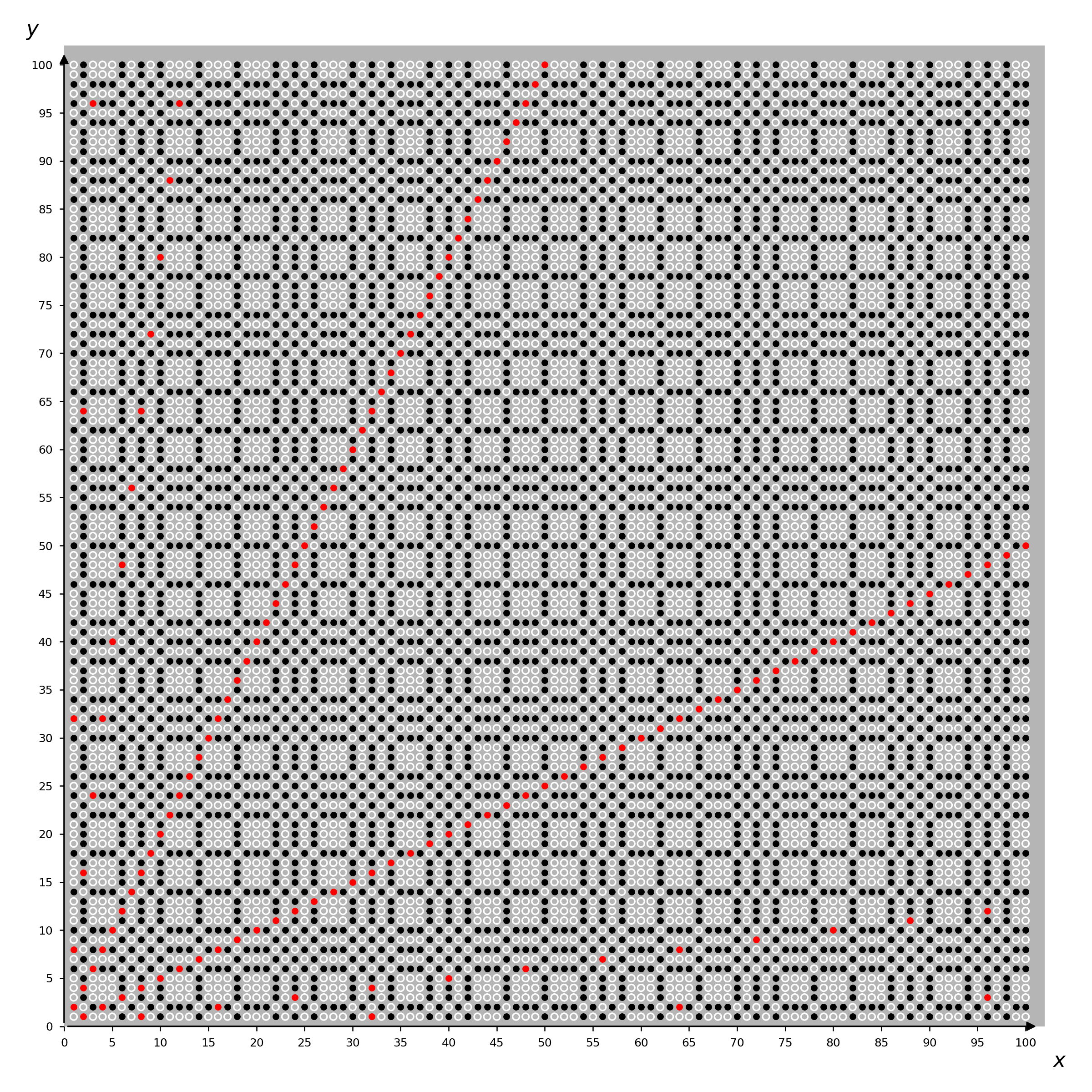}}
\caption{The refined parameter tiling: white corresponds to Cantor sets, black to Cantorvals, and red to finite unions of closed intervals.}
\label{tilingdet}
\end{figure}

\noindent \textbf{Organisation of the paper.} In \S \ref{persp} we describe the three distinct ways in which the sets $E(x,y)$ may be viewed. In \S \ref{proofs} we prove Theorems  \ref{main} and \ref{refinement}. In \S \ref{tilingprops} we study the aperiodic order of the tiling $T$. Finally, in \S \ref{final}, we place our results in a broader context, discussing their motivation, connections with classical parameter spaces, and some final questions and remarks.

\section{Equivalent representations of $E(x,y)$}\label{persp}
In this section we begin by providing three equivalent representations of $E(x,y)$: first as a self-similar set, secondly as an achievement set of a bi-geometric series, and thirdly as a `rational' projection of the four-corner Cantor set.
\subsection{Representation as a self-similar set}

Consider the iterated function system

$$
\mathcal{F}_{x,y}
=
(f_0,f_x,f_y,f_{x+y}),
\qquad
(x,y)\in\mathbb{N}^2,
$$

where

$$
f_d(z)=\frac{z+d}{4},
\qquad
d\in\{0,x,y,x+y\}.
$$
Throughout the paper, statements about the similarity dimension, exact overlaps, and the open set condition for $\mathcal{F}_{x,y}$ refer to this iterated function system (IFS).

For each \((x,y)\in\mathbb{N}^2\), the maps in \(\mathcal{F}_{x,y}\) are similarities of \(\mathbb{R}\) with common contraction ratio \(1/4\). Hence there exists a unique non-empty compact attractor \(E(x,y)\subset\mathbb{R}\) satisfying

$$
E(x,y) = \bigcup_{f \in \mathcal{F}_{x,y}} f(E(x,y)).
$$

The similarity dimension of the iterated function system $\mathcal{F}_{x,y}$ is $1$ for every $(x,y)\in\mathbb{N}^2$. We say that $(x,y)$ is a dimension-drop parameter if $\dim_H E(x,y)<1$.
We say that $\mathcal{F}_{x,y}$ satisfies the open set condition (OSC) if there exists a non-empty open set $O$ such that 
$$f_0(O) \cup f_x(O) \cup f_y(O) \cup f_{x+y}(O) \subset O$$ 
and such that the sets in the union are pairwise disjoint. In the diagonal case $x=y$, we have $f_x=f_y$, so the iterated function system has an exact overlap at the first level and does not satisfy the open set condition.

\subsection{Representation as a set of subsums}
The set
$$E(u_n)=\left\{ \sum_{n=1}^{\infty} \varepsilon_n u_n : (\varepsilon_n) \in \{ 0, 1\}^{\mathbb{N}}\right\}$$
is referred to as the achievement set or the set of subsums of the series $\sum u_n$. 
The study of achievement sets was initiated by Kakeya~\cite{kakeya} and remains an active area of research; see, for example,~\cite{GP25} for a list of open problems.

A fundamental result in this theory is the following trichotomy~\cite{NS}: the achievement set of a convergent positive series is either
\begin{enumerate}
 \item a finite union of closed intervals;
 \item a Cantor set;
 \item a Cantorval.
\end{enumerate}

Determining necessary and sufficient conditions for an achievement set to belong to one of these three classes is difficult in general. Particularly well studied are \emph{multigeometric series} of the form
\begin{equation}
\label{MGS}
\sum_{n=1}^{\infty}{z_n}=k_1q+k_2q+\dots+k_mq+k_1q^2+\dots +k_mq^2+\dots+k_1q^{i}+\dots+k_mq^{i}+\dots,
\end{equation}
where $k_1, k_2, \dots, k_m$ are fixed positive scalars, and $0<q<1$. In particular, conditions for the set of subsums of a multigeometric series to be a Cantorval were established in \cite{Banakh,Jones,Ferdinands,BP17,Bartoszewicz} with further generalisations provided in \cite{KMV}.

We will repeatedly use the following elementary properties of multigeometric series.
 
\begin{lemma}
\label{L1}
If for the series \eqref{MGS} the inequality $z_{j} \geq z_{j+1}$ ($z_j < z_{j+1}$) holds for some $1 \leq j \leq m$, then the inequality $z_{j+im} \geq z_{j+1+im}$ ($z_{j+im} < z_{j+1+im}$) also holds for every $i \in \mathbb{N}$.
\end{lemma}

\begin{lemma}
\label{L2}
If for the series \eqref{MGS} the relation $z_j \leq Z_{j}=\sum_{n=j+1}^{\infty}z_n$ $(z_{j} > Z_{j})$  holds for some $1 \leq j \leq m$, then the relation $z_{j+im} \leq Z_{j+im}$ $(z_{j+im} > Z_{j+im})$  also holds for every $i \in \mathbb{N}$.
\end{lemma}

For each $(x,y) \in \N^2$ we have
$$E(  x, y)=\left\{ \sum_{i=1}^{\infty} \frac{\varepsilon_i}{4^{i}} : (\varepsilon_i) \in \{0,  x, y,   x+y\}^{\mathbb{N}}\right\}.$$
In particular, this makes it the set of subsums of the bi-geometric series

\begin{equation*}
\sum_{n=1}^{\infty}a_n=\frac{y}{4} + \frac{x}{4} +\frac{y}{4^2}+\frac{x}{4^2}+ \cdots + \frac{y}{4^i}+\frac{x}{4^i}+ \cdots,
\end{equation*}
which is a particular case of the series \eqref{MGS} with $m=2, k_1=y, k_2=x, q = 1/4$.

The trichotomy above therefore implies that, for every $(x,y)\in\mathbb N^2$, the set $E(x,y)$ is either a finite union of closed intervals, a Cantor set, or a Cantorval. Our aim is to determine explicitly which of these three possibilities occurs for each pair of parameters.

\subsection{Representation as a projection of the four-corner Cantor set}

The set $E(x,y)$ also admits a natural geometric representation as
a linear projection of the four-corner Cantor set. Let
\[
C_{1/4}
=
\left\{
\sum_{n=1}^{\infty}\frac{3\varepsilon_n}{4^n}:
\varepsilon_n\in\{0,1\}
\right\}
\]
and consider the four-corner Cantor set
\[
C=C_{1/4}\times C_{1/4}.
\]
For $\tau>0$, define the linear map
\[
\pi_\tau(u,v)=\tau u+v.
\]
Then
\[
\pi_\tau\left(\frac{1}{3}C\right)
=
\left\{
\sum_{n=1}^{\infty}
\frac{\tau\varepsilon_n+\eta_n}{4^n}:
\varepsilon_n,\eta_n\in\{0,1\}
\right\}.
\]
Since
\[
\tau\varepsilon_n+\eta_n\in\{0,1,\tau,\tau+1\},
\]
we obtain
\[
\pi_\tau\left(\frac{1}{3}C\right)
=
\left\{
\sum_{n=1}^{\infty}\frac{\delta_n}{4^n}:
\delta_n\in\{0,1,\tau,\tau+1\}
\right\}.
\]

Taking $\tau=x/y$, we arrive at the identity
\begin{equation}
\label{eq:E_projection}
E(x,y)
=
y\pi_{x/y}\left(\frac{1}{3}C\right).
\end{equation}
Indeed,
\[
y\pi_{x/y}\left(\frac{1}{3}C\right)
=
\left\{
\sum_{n=1}^{\infty}
\frac{x\varepsilon_n+y\eta_n}{4^n}:
\varepsilon_n,\eta_n\in\{0,1\}
\right\},
\]
and the possible digits $x\varepsilon_n+y\eta_n$ are precisely
\[
0,\quad x,\quad y,\quad x+y.
\]

This representation can also be expressed in terms of orthogonal
projections. Let $P_\tau$ denote the scalar orthogonal projection onto
the direction of the vector $(\tau,1)$, that is,
\[
P_\tau(u,v)=\frac{\tau u+v}{\sqrt{1+\tau^2}}.
\]
Since
\[
\pi_\tau=\sqrt{1+\tau^2}\,P_\tau,
\]
identity \eqref{eq:E_projection} can be rewritten as
\[
E(x,y)
=
\sqrt{x^2+y^2}\,
P_{x/y}\left(\frac{1}{3}C\right).
\]
Thus, $E(x,y)$ is a homothetic image of an orthogonal projection of
the four-corner Cantor set in the rational direction determined by
the vector $(x,y)$.

This places our problem within the classical theory of projections of the four-corner Cantor set. In particular, Theorem \ref{refinement} may be viewed as a refinement of \cite[Theorem~10.5]{Mattila}. 
\section{Topological classification of $E(x,y)$} \label{proofs}
We begin by justifying Table \ref{tab} by showing that there are several equivalent ways to characterise the tiling $T$.

\begin{proposition}[Equivalent characterisations]
\label{tableprop}
For $(x,y)\in\mathbb{N}^2$, the following are equivalent:
\begin{enumerate}
    \item $T(x,y)=1$;
    \item $E(x,y)$ has non-empty interior;
    \item $E(x,y)$ has positive Lebesgue measure;
    \item $\dim_H E(x,y)=1$;
    \item the semigroup generated by $\mathcal{F}_{x,y}$ is free;
    \item $\mathcal{F}_{x,y}$ satisfies the open set condition.
\end{enumerate}
\end{proposition}

\begin{proof}
The equivalence of (1)-(4) follow directly from Theorem \ref{refinement}. The equivalence to (5) is because the exact overlaps conjecture holds for IFSs with algebraic parameters, by \cite{Hochman}. Finally, to see (6), note that by Schief~\cite{Schief}, a self-similar set with similarity dimension $s=1$ has positive Lebesgue measure if and only if the OSC is satisfied.
\end{proof}

\subsection{Parameter reductions}
There are several reductions that will be useful.
From now on, \(A\sim B\) means that the sets \(A\) and \(B\) either both have empty interior or both have non-empty interior.

\begin{lemma}

\label{T:Divisibility of one parameter}
For all $x, y \in \mathbb{N}$, the following relations hold:
\begin{enumerate}[label=\textup{(\roman*)}]
\item $E(x,y) \sim E(y,x)$;
\item $E(x,y) \sim E(\eta x, \eta y)$ for any $\eta \in \mathbb{R} \setminus \{0\}$;
\item $E(x,y) \sim E_{\nu}(x,y) = \left\{ \sum_{i=1}^{\infty} \frac{\varepsilon_i}{4^{i}} : (\varepsilon_i) \in \{\nu, x+\nu, y+\nu,   x+y+\nu\}^{\mathbb{N}}\right\}$ for any $\nu \in \mathbb{R}$;
\item $E(x,y) \sim E(4^lx, 4^my)$ for any $l, m \in \mathbb{N}$.
\end{enumerate}

\end{lemma}

\begin{proof}

Since the set of subsums remains unchanged under any rearrangement of the series terms, \(E(x,y)=E(y,x)\), which proves part~\textup{(i)}. This also implies that the parameter tiling is symmetric with respect to the line \(y=x\).

Since \(E(\eta x,\eta y)=\eta E(x,y),\) the set \(E(\eta x,\eta y)\) is a scaled copy of \(E(x,y)\). As every non-zero scaling preserves the presence or absence of interior, part~\textup{(ii)} follows. Therefore, if \(x\) and \(y\) are not coprime and \(d=\gcd(x,y)\), then 
\(
E(x,y)
=
dE\left(\frac{x}{d},\frac{y}{d}\right),
\)
and hence
\(
E(x,y)
\sim
E\left(\frac{x}{d},\frac{y}{d}\right).
\)
Thus, we may restrict our attention to the case in which \(x\) and \(y\) are coprime.

Every digit
\(
\varepsilon_i\in\{\nu,x+\nu,y+\nu,x+y+\nu\}
\)
can be written as
\(
\varepsilon_i=\delta_i+\nu,
\)
where
\(
\delta_i\in\{0,x,y,x+y\}.
\)
Therefore,
\[
\begin{aligned}
E_{\nu}(x,y)
&=
\left\{
\sum_{i=1}^{\infty}
\frac{\delta_i+\nu}{4^i}:
(\delta_i)\in\{0,x,y,x+y\}^{\mathbb N}
\right\}
=
E(x,y)
+
\nu\sum_{i=1}^{\infty}\frac{1}{4^i}
=
E(x,y)+\frac{\nu}{3}.
\end{aligned}
\]
Thus, \(E_{\nu}(x,y)\) is an isometric copy of \(E(x,y)\) with translation vector \(\nu/3\). Hence,
\(
E(x,y)\sim E_{\nu}(x,y),
\)
which proves part~\textup{(iii)}.

To prove (iv), let us consider the bi-geometric series with initial parameters $4^{l}x$ and $y$, where $l \in \mathbb{N}$.
We have
$$\sum_{n=1}^{\infty}{a_n}=\frac{y}{4}+\frac{4^l x}{4}+\frac{y}{4^2}+\frac{4^l x}{4^2}+\dots+\frac{y}{4^l}+\frac{4^l x}{4^l}+\dots +\frac{y}{4^i}+ \frac{4^l x}{4^i}+\dots=$$
$$=4^{l-1}x+4^{l-2}x+ \dots +4 x + x +\frac{y}{4}+\frac{x}{4}+ \dots +\frac{y}{4^i}+\frac{x}{4^i}+\dots,$$
where the terms in the second line have been rearranged and reindexed.
We observe that the set of subsums of $\sum_{n=l+1}^{\infty}{a_n}$ coincides with $E(x, y)$. Therefore, the set $E(4^{l}x,y)$ can be represented as the
arithmetic (Minkowski) sum $E(4^{l}x,y) = L \oplus E(x,y)$, where
$$L=\left\{ \sum_{n=1}^{l} \varepsilon_n 4^{n-1}x : (\varepsilon_n) \in \{0, 1 \}^{l} \right\}$$
is a finite set of points. Equivalently,
\[
E(4^lx,y)
=
\bigcup_{\lambda\in L}\bigl(\lambda+E(x,y)\bigr).
\]

Since \(E(x,y)\) is compact, all the sets in this finite union are closed. By the Baire category theorem, \(E(4^lx,y)\) has non-empty interior if and only if at least one of these translates has non-empty interior. Since translations preserve the presence or absence of interior, we obtain
\(
E(4^lx,y)\sim E(x,y).
\)

Analogously, we can show that $E(x, 4^my) \sim E(x,y)$ for any $m \in \mathbb{N}$.
By combining these two cases, we conclude that $E(4^{l}x,4^m y) \sim E(x,y)$.

\end{proof}

\subsection{The empty-interior case}

For a positive integer $n$, let $n^*\in\{1,2,3\}$ denote the last
non-zero digit in its quaternary expansion. Recall \cite[Theorem~10.5(b)]{Mattila}, which states that, for coprime
positive integers $p$ and $q$,
\[
\mathcal L^1\left(\pi_{p/q}(C)\right)=0
\qquad\text{and}\qquad
\dim_{\mathrm H}\pi_{p/q}(C)<1
\]
whenever both $p^*$ and $q^*$ are odd.

\begin{theorem}
\label{thm:empty_interior}
If $x$ and $y$ belong to the same class, either both to $V$ or both
to $D$, then
\[
\operatorname{int}E(x,y)=\varnothing.
\]
Moreover,
\[
\mathcal L^1(E(x,y))=0
\qquad\text{and}\qquad
\dim_{\mathrm H}E(x,y)<1.
\]
Consequently, $E(x,y)$ is a Cantor set.
\end{theorem}

\begin{proof}
Let
\[
g=\gcd(x,y),
\qquad
p=\frac{x}{g},
\qquad
q=\frac{y}{g}.
\]
Then $\gcd(p,q)=1$ and $x/y=p/q$.

Let $\nu_2(n)$ denote the exponent of $2$ in the prime factorisation
of $n$. By the definitions of $V$ and $D$,
\[
n\in V
\quad\Longleftrightarrow\quad
\nu_2(n)\ \text{is even},
\]
whereas
\[
n\in D
\quad\Longleftrightarrow\quad
\nu_2(n)\ \text{is odd}.
\]
Since
\[
\nu_2(x)=\nu_2(g)+\nu_2(p),
\qquad
\nu_2(y)=\nu_2(g)+\nu_2(q),
\]
division by $g$ preserves the relation of belonging to the same class.
Thus, $p$ and $q$ belong to the same class. Since they are coprime,
they cannot both belong to $D$, and hence $p^*$ and $q^*$ are both odd.

By \cite[Theorem~10.5(b)]{Mattila},
\[
\mathcal L^1\left(\pi_{p/q}(C)\right)=0
\qquad\text{and}\qquad
\dim_{\mathrm H}\pi_{p/q}(C)<1.
\]
The projection representation \eqref{eq:E_projection} gives
\[
E(x,y)
=
\frac{y}{3}\pi_{p/q}(C).
\]
Since non-zero homotheties preserve zero Lebesgue measure and
Hausdorff dimension, it follows that
\[
\mathcal L^1(E(x,y))=0
\qquad\text{and}\qquad
\dim_{\mathrm H}E(x,y)<1.
\]
Therefore, $E(x,y)$ has empty interior. Finally, the topological
trichotomy for achievement sets implies that $E(x,y)$ is a Cantor set.
\end{proof}

\subsection{The non-empty interior case}

We now use the complementary part of
\cite[Theorem~10.5(c)]{Mattila}. It states that, for coprime positive
integers $p$ and $q$, the projection $\pi_{p/q}(C)$ contains a
non-degenerate interval whenever either $p^*$ or $q^*$ is even.
Moreover,
\[
\pi_{p/q}(C)
=
\overline{\operatorname{int}\pi_{p/q}(C)}.
\]

\begin{theorem}
\label{thm:nonempty_interior}
If one of $x,y$ belongs to $V$ and the other belongs to $D$, then
$E(x,y)$ contains a non-degenerate interval. In particular,
\[
\operatorname{int}E(x,y)\neq\varnothing
\]
and
\[
E(x,y)=\overline{\operatorname{int}E(x,y)}.
\]
Consequently, $E(x,y)$ is either a finite union of closed intervals
or a Cantorval.
\end{theorem}

\begin{proof}
Let
\[
g=\gcd(x,y),
\qquad
p=\frac{x}{g},
\qquad
q=\frac{y}{g}.
\]
Then $\gcd(p,q)=1$ and $x/y=p/q$.

As in the proof of Theorem~\ref{thm:empty_interior}, division by the
common factor $g$ preserves whether the two parameters belong to the
same or to different classes $V$ and $D$. Hence, $p$ and $q$ belong
to different classes.

Since $p$ and $q$ are coprime, at least one of them is odd. Therefore,
exactly one of the last non-zero quaternary digits $p^*$ and $q^*$ is
even. By \cite[Theorem~10.5(c)]{Mattila}, the projection
$\pi_{p/q}(C)$ contains a non-degenerate interval and satisfies
\[
\pi_{p/q}(C)
=
\overline{\operatorname{int}\pi_{p/q}(C)}.
\]

By the projection representation \eqref{eq:E_projection},
\[
E(x,y)
=
\frac{y}{3}\pi_{p/q}(C).
\]
Since a non-zero homothety preserves the presence of intervals and
commutes with taking the interior and its closure, we conclude that
$E(x,y)$ contains a non-degenerate interval and
\[
E(x,y)=\overline{\operatorname{int}E(x,y)}.
\]

Finally, the topological trichotomy for achievement sets implies that
$E(x,y)$ is either a finite union of closed intervals or a Cantorval.
\end{proof}

\subsection{Classification of $E(x,y)$ which is a finite union of intervals}

For a convergent positive series $\sum u_n$ and $k \in \mathbb{N}$, we define its $k$-tail as
$$U_{k}:=\sum^{\infty}_{n=k+1} u_n.$$
By $E(u_n)$ we denote the set of subsums of the series. We say that a series is \textbf{ordered} if its terms form a non-increasing sequence, i.e., $u_n \geq u_{n+1}$ for every $n \in \mathbb{N}$.

For ordered series, classical results of Kakeya, Hornich, and Menon \cite{kakeya, Hornich, Menon} provide a necessary and sufficient condition for the set of subsums to be a finite union of closed intervals, as well as a sufficient condition for it to be homeomorphic to the Cantor set.

\begin{theorem} (Kakeya--Hornich--Menon)
\label{KHM}

For an ordered convergent positive series, $E(u_n)$ is:
\begin{enumerate}
\item\label{thm:KHM.3} a closed interval if and only if $u_n \leq U_n$ for all $n \in \mathbb{N}$;

\item\label{thm:KHM.2} a finite union of closed intervals if and only if $u_{n} \leq U_{n}$ holds for all sufficiently large $n$;

\item\label{thm:KHM.4} homeomorphic to the Cantor set if $u_{n} > U_n$ holds for all sufficiently large $n$.
\end{enumerate}
\end{theorem}

As we mentioned in Section 2, $E(x,y)$ coincides with the set $E(a_n)$ of subsums of the bi-geometric series
\begin{equation}
\label{SBGS}
\sum_{n=1}^{\infty}a_n=\frac{y}{4} + \frac{x}{4} +\frac{y}{4^2}+\frac{x}{4^2}+ \cdots + \frac{y}{4^i}+\frac{x}{4^i}+ \cdots,
\end{equation}
which is a special case of multigeometric series \eqref{MGS}.

Since \(E(x,y)=E(y,x)\), we may assume without loss of generality that \(y\geq x\). If $x=y$, then $x$ and $y$ belong to the same set, either $V$ or $D$. Hence, by Theorem~\ref{thm:empty_interior}, $E(x,y)$ is a Cantor set and therefore cannot be a finite union of closed intervals. Thus, in what follows, we may assume that $y>x$. The case where $x>y$ is symmetric with respect to the line $y=x$ in the coordinate plane.

In order to apply Theorem \ref{KHM}, we must first solve the problem of ordering the terms of the series \eqref{SBGS}.

\begin{lemma}
\label{OST}
Let \(y>x\), and define
\begin{equation}
\label{Definition of t}
t:=\left\lceil \log_{4} \frac{y}{x} \right\rceil - 1.
\end{equation}
Then \(t\in\mathbb{N}_0\), and the non-increasing rearrangement of the series \eqref{SBGS} is given by
\begin{equation}
\label{Ordered Series}
\sum_{n=1}^{\infty}a'_n
=
\sum_{j=1}^{t}\frac{y}{4^j}
+
\sum_{i=1}^{\infty}\left(\frac{y}{4^{t+i}}+\frac{x}{4^i}\right),
\end{equation}
where the first sum is understood to be empty when \(t=0\).
\end{lemma}
\begin{proof}
Since \(y/x>1\), we have \(t\in\mathbb{N}_0\). Moreover, the definition of \(t\) gives
\[
t<\log_4\frac{y}{x}\leq t+1,
\]
and hence
\[
4^t<\frac{y}{x}\leq 4^{t+1}.
\]
Consequently, for every \(i\in\mathbb{N}\),
\[
\frac{y}{4^{t+i}}
>
\frac{x}{4^i}
\geq
\frac{y}{4^{t+i+1}}.
\]
When \(t\geq1\), the terms
\[
\frac{y}{4},\ldots,\frac{y}{4^t}
\]
form an initial decreasing block; when \(t=0\), this block is empty. In either case, the remaining terms alternate in non-increasing order as
\[
\frac{y}{4^{t+1}},\frac{x}{4},
\frac{y}{4^{t+2}},\frac{x}{4^2},\ldots.
\]
This proves the stated representation.
\end{proof}

\begin{theorem}
\label{FUIT}
$E(x, y)$ is a finite union of intervals if and only if $y=2 \cdot 4^{t} \cdot x$ for some $t \in \mathbb{Z}$.
\end{theorem}

\begin{proof}

Note that the ordered series \eqref{Ordered Series} is bi-geometric starting from the $(t+1)$-th term, i.e., it admits the form
$$\sum_{n=t+1}^{\infty} a_n'=k_1q+k_2q+k_1q^2+k_2q^2+\dots+k_1 q^i+k_2q^i+\dots,$$
with parameters $k_1=y/4^t, k_2=x$ and $q=1/4$. The Kakeya-Hornich-Menon criterion for being a finite union of closed intervals depends only on the inequalities \(a_n\leq A_n\) for all sufficiently large \(n\). Therefore, it is sufficient to analyse the bi-geometric tail.

Taking into account Lemma \ref{L2} and Theorem \ref{KHM}, $E(a'_n)$ (and accordingly $E(a_n)$) will be a finite union of intervals if and only if

$$
\left\{
\begin{array}{l}
a_{t+1}' \leq A_{t+1}'=a'_{t+2}+a'_{t+3}+a'_{t+4}+\dots, \\
a_{t+2}' \leq A_{t+2}'=a'_{t+3}+a'_{t+4}+\dots.\\
\end{array}
\right.
$$
The latter system can be rewritten as
$$
\left\{
\begin{array}{l}
\frac{y}{4^{t+1}} \leq \frac{x}{4}+\frac{y}{4^{t+2}}+\frac{x}{4^2}+\frac{y}{4^{t+3}}+\dots, \\
\frac{x}{4} \leq \frac{y}{4^{t+2}}+\frac{x}{4^2}+\frac{y}{4^{t+3}}+\frac{x}{4^3}+\dots,\\
\end{array}
\right.
$$
from which it follows that

\[
\left\{
\begin{array}{l}
\frac{y}{4^{t+1}} \leq \sum_{i=1}^{\infty} {\frac{x}{4^{i}}}+\sum_{i=1}^{\infty} {\frac{y}{4^{t+1+i}}}, \\
\frac{x}{4} \leq \sum_{i=1}^{\infty} {\frac{x}{4^{i+1}}}+\sum_{i=1}^{\infty} {\frac{y}{4^{t+1+i}}},\\
\end{array}
\right.
\Leftrightarrow
\left\{
\begin{array}{l}
y \leq 2 \cdot 4^t \cdot x, \\
y \geq 2 \cdot 4^t \cdot x,\\
\end{array}
\right.
\]
This system of inequalities has a solution only if $y=2 \cdot 4^{t} \cdot x$  for some $t \in \mathbb{N}_0$. Adding the symmetric solution $y=x / (2 \cdot 4^{t})$ for some $t \in \mathbb{N}_0$ completes the proof of the theorem.

\end{proof}

\begin{corollary}
$E(x, y)$ is the interval $\left[0, \frac{x+y}{3}\right]$ if and only if $y=2x$ or $y=x/2$.
\end{corollary}
\begin{proof}

First of all, note that under the condition $\frac{y}{4}>x$ the initial series \eqref{SBGS} is not ordered. Moreover, for the corresponding ordered series \eqref{Ordered Series}, the inequality $a'_i > A'_i$ will hold for every $1 \leq i \leq t$. Taking into account Theorem \ref{KHM} \eqref{thm:KHM.3}, $E(x, y)$ is not an interval.

In the case $\frac{y}{4} \leq x$, the series \eqref{SBGS} is an ordered bi-geometric series from the very beginning. According to Theorem \ref{KHM} \eqref{thm:KHM.3}, for $E(x, y)$ to be an interval, it is necessary and sufficient that the condition

\[
\left\{
\begin{array}{l}
a_1 \leq A_1, \\
a_2 \leq A_2,\\
\end{array}
\right.
\Leftrightarrow
\left\{
\begin{array}{l}
\frac{y}{4} \leq \frac{x}{4}+\frac{y}{4^2}+\frac{x}{4^2}+\frac{y}{4^3}+\frac{x}{4^3}+\dots, \\
\frac{x}{4} \leq \frac{y}{4^2}+\frac{x}{4^2}+\frac{y}{4^3}+\frac{x}{4^3}+\dots,\\
\end{array}
\right.
\Leftrightarrow
\left\{
\begin{array}{l}
y \leq 2x, \\
y \geq 2x,\\
\end{array}
\right.
\]
holds. This system of inequalities has a solution if and only if $y=2x$. Adding the symmetric solution $x=2y$ completes the proof of the corollary.

\end{proof}

\begin{corollary}
If $y=2 \cdot 4^{t} \cdot x$ for some $t \in \mathbb{Z}$, then $E(x,y)$ consists of exactly $2^{\alpha}$ pairwise disjoint closed intervals, where
$$
\alpha=
\left\{
\begin{array}{r}
t, ~ \text{if} ~ ~ t \geq 0,\\
-t-1, ~ \text{if} ~ ~ t < 0.\\

\end{array}
\right.
$$
\end{corollary}

\begin{proof}
First suppose that $y=2 \cdot 4^t \cdot x$ with $t\geq0$. Separating the
first $t$ terms involving $y$ in the bi-geometric series, we obtain
\[
E(x,y)=L_t\oplus E(x,2x),\qquad
L_t=\left\{2x\sum_{j=0}^{t-1}\varepsilon_j4^j:
\varepsilon_j\in\{0,1\}\right\},
\]
where $L_0=\{0\}$. By the preceding corollary,
$E(x,2x)=[0,x]$. The set $L_t$ contains exactly $2^t$ points,
and any two distinct points of $L_t$ are at least $2x$ apart.
Hence the intervals $\lambda+[0,x]$, $\lambda\in L_t$,
are pairwise disjoint.

If $t<0$, write $\alpha=-t-1\geq0$. Then $x=2\cdot4^\alpha \cdot y$,
so the same argument applies after interchanging $x$ and $y$.
The number of intervals is therefore $2^\alpha=2^{-t-1}$.
\end{proof}

\section{Properties of the tiling} \label{tilingprops}

In this section we establish the arithmetic and dynamical properties of the
parameter tiling. We first study the characteristic sequence of the
partition $\mathbb{N}=V\sqcup D$, and then use its properties to derive
corresponding properties of the tiling.

\subsection{Properties of the characteristic sequence}

Recall that the characteristic function
\[
c\colon\mathbb{N}\longrightarrow\{0,1\}
\]
of the set \(V\) is defined by
\[
c(n)=
\begin{cases}
1, & \text{if }n\in V,\\
0, & \text{if }n\in D.
\end{cases}
\]
Setting \(c_n:=c(n)\), we call the sequence $(c_n)_{n=1}^{\infty}$ the characteristic sequence of \(V\).

We will use the following arithmetic properties of $V$ and $D$ to study $(c_n)$.

\begin{lemma}
\label{Properties_V}
For every $n\in\mathbb{N}$,
\begin{enumerate}[label=(\roman*)]
    \item $n\in V$ if and only if $2n\in D$;
    \item $n\in V$ if and only if $4n\in V$;
    \item $n\in D$ if and only if $4n\in D$;
    \item every odd positive integer belongs to $V$;
    \item $4n-2\in D$.
\end{enumerate}
\end{lemma}

\begin{proof}
Write
\[
n=4^k m,
\]
where \(k\in\mathbb{N}_0\) and \(m\) is not divisible by \(4\). The last non-zero digit in the quaternary expansion of \(n\) is determined by the residue of \(m\) modulo \(4\). Thus,
\[
n\in V
\quad\Longleftrightarrow\quad
m\equiv1\ \text{or}\ 3\pmod 4,
\]
whereas
\[
n\in D
\quad\Longleftrightarrow\quad
m\equiv2\pmod 4.
\]
If \(m\) is odd, then the last non-zero quaternary digit of \(2n\) is \(2\), and hence \(2n\in D\). If \(m\equiv2\pmod4\), write \(m=2r\), where \(r\) is odd. Then
\[
2n=4^{k+1}r,
\]
so \(2n\in V\). This proves statement~(i). Multiplication by \(4\) leaves the last non-zero quaternary digit unchanged, which proves statements~(ii) and~(iii).

Every odd positive integer is congruent to either \(1\) or \(3\) modulo \(4\), proving statement~(iv). Finally,
\[
4n-2\equiv2\pmod4,
\]
and hence \(4n-2\in D\), proving statement~(v).
\end{proof}

\begin{corollary}
\label{cor:characteristic-recursion}
For any $n \in \N$, the characteristic sequence satisfies
\[
c_{4n-3}=1,\qquad
c_{4n-2}=0,\qquad
c_{4n-1}=1,\qquad
c_{4n}=c_n.
\]
\end{corollary}

\begin{proof}
The first three identities follow from parts~(iv) and~(v) of
Lemma~\ref{Properties_V}. The final identity follows from
parts~(ii) and~(iii).
\end{proof}

By the previous corollary, the sequence \((c_n)\) can be constructed recursively as follows. We start with the periodic pattern
\[
1\ 0\ 1\ ?\ 
1\ 0\ 1\ ?\ 
1\ 0\ 1\ ?\ 
1\ 0\ 1\ ?\
1\ 0\ 1\ ?\
1\ 0\ 1\ ?\ \cdots
\]
and fill the successive placeholders \(?\) with the terms \(c_1,c_2,c_3,\ldots\). This yields
\[
1\ 0\ 1\ 1\ 
1\ 0\ 1\ 0\ 
1\ 0\ 1\ 1\ 
1\ 0\ 1\ 1\ 
1\ 0\ 1\ 1\ 
1\ 0\ 1\ 0\ \cdots
\]
This recursive procedure is a Toeplitz construction: at each stage, the
positions not yet determined are filled recursively by the original
sequence.

\begin{lemma}
The characteristic sequence \((c_n)\) is a fixed point of the quaternary substitution
\[
1 \mapsto 1 \ 0 \ 1 \ 1,\qquad
0 \mapsto 1 \ 0 \ 1 \ 0.
\]
\end{lemma}
\begin{proof}
By Corollary~\ref{cor:characteristic-recursion}, 
\[
\bigl(c_{4n-3},c_{4n-2},c_{4n-1},c_{4n}\bigr)
=
(1,0,1,c_n).
\]
Thus, if \(c_n=1\), the corresponding block is \(1\ 0\ 1\ 1\), whereas if \(c_n=0\), it is \(1\ 0\ 1\ 0\). Consequently, applying the substitution to
$(c_n)$ reproduces the same sequence, proving that $(c_n)$ is a fixed
point of the substitution.
\end{proof}

\begin{corollary}
\label{non-periodic_c(n)}
The characteristic sequence \((c_n)\) is not periodic.
\end{corollary}
\begin{proof}
Suppose, to the contrary, that \((c_n)\) has a positive integer period \(t\). Then
\[
c_{n+t}=c_n
\]
for every \(n\in\mathbb{N}\). Taking \(n=t\), we obtain
\[
c_{2t}=c_t.
\]
On the other hand, statement~(i) of Lemma~\ref{Properties_V} implies that exactly one of \(t\) and \(2t\) belongs to \(V\). Consequently,
\[
c_{2t}=1-c_t,
\]
which contradicts \(c_{2t}=c_t\). Therefore, \((c_n)\) is not periodic.
\end{proof}

\subsection{Properties of the tiling $T$}

We first record the symmetry of the parameter tiling.
\begin{proposition}
The tiling is symmetric with respect to the main diagonal \(y=x\).
\end{proposition}
\begin{proof}
The result follows immediately from the equality
\[
T(x,y)=T(y,x)
\]
for all \(x,y\in\mathbb{N}\). Indeed, the condition \(c(x)\neq c(y)\) is symmetric in \(x\) and \(y\).
\end{proof}

\begin{lemma}
Each row and each column of the tiling is either the characteristic sequence of \(V\) or its complement.
\end{lemma}

\begin{proof}
Fix \(x\in\mathbb{N}\). If \(x\in D\), then \(c(x)=0\), and therefore
\[
T(x,y)=c(y)
\]
for every \(y\in\mathbb{N}\). Hence, the corresponding row is the characteristic sequence \((c_n)\).

If \(x\in V\), then \(c(x)=1\), and therefore
\[
T(x,y)=1-c(y)
\]
for every \(y\in\mathbb{N}\). Hence, the corresponding row is the complementary sequence \((1-c_n)\).

The same argument, with \(x\) and \(y\) interchanged, applies to every column. This completes the proof.
\end{proof}

\begin{theorem}\label{aperiodicthm} 
The parameter tiling $T$ is aperiodic; that is, there is no non-zero translation vector \((a,b)\in\mathbb{N}_0^2\) such that
\[
T(x+a,y+b)=T(x,y)
\]
for all \(x,y\in\mathbb{N}\).
\end{theorem}
\begin{proof}
Suppose, to the contrary, that the tiling is invariant under a non-zero translation vector
\[
(a,b)\in\mathbb{N}_0^2.
\]
If \(a=0\), then \(b>0\). By the symmetry of the tiling with respect to the main diagonal, invariance under \((0,b)\) implies invariance under \((b,0)\). Thus, after interchanging the coordinates if necessary, we may assume that \(a>0\).

Fix \(y'\in D\). Since \(c(y')=0\), we have
\[
T(x,y')=c(x)
\]
for every \(x\in\mathbb{N}\). We shall show that
\[
T(x,y')\neq T(x+a,y'+b)
\]
for some \(x\in\mathbb{N}\). There are two cases.

\begin{enumerate}[label=(\roman*)]
\item Suppose that \(y'+b\in D\). Then
\[
T(x+a,y'+b)=c(x+a)
\]
for every \(x\in\mathbb{N}\). Translational invariance would therefore imply
\[
c(x)=c(x+a)
\]
for every \(x\in\mathbb{N}\), making \(a\) a positive period of the characteristic sequence. This contradicts Corollary~\ref{non-periodic_c(n)}. In particular, taking \(x=a\), as we saw in Lemma~\ref{Properties_V}(i),
\[
c(a)\neq c(2a).
\]

\item Suppose that \(y'+b\in V\). Then
\[
T(x+a,y'+b)=1-c(x+a)
\]
for every \(x\in\mathbb{N}\). Thus, translational invariance would imply
\[
c(x)=1-c(x+a)
\]
for every \(x\in\mathbb{N}\).

If \(a\) is even, take \(x=1\). Both \(1\) and \(1+a\) are odd and therefore belong to \(V\). Hence,
\[
c(1)=c(1+a)=1,
\]
which contradicts \(c(1)=1-c(1+a)\).

If \(a\) is odd, take \(x=4\). We have \(4\in V\), while \(4+a\) is odd and therefore also belongs to \(V\). Consequently,
\[
c(4)=c(4+a)=1,
\]
which contradicts \(c(4)=1-c(4+a)\).
\end{enumerate}

Both cases lead to a contradiction. Therefore, the tiling has no non-zero translation vector in \(\mathbb{N}_0^2\) and is aperiodic.
\end{proof}

\begin{theorem}\label{subthm}
The parameter tiling $T$ is a fixed point of the substitution defined as follows:

\begin{center}
\begin{tabular}{cc}
\(\displaystyle
V\times V \ni (x,y) \mapsto
\begin{array}{r}
\circ \bullet \circ \circ \\[-1.5ex]
\circ \bullet \circ \circ \\[-1.5ex]
\bullet \circ \bullet \bullet \\[-1.5ex]
\circ \bullet \circ \circ
\end{array}
\)
& \quad \quad
\(\displaystyle
V\times D \ni (x,y) \mapsto
\begin{array}{r}
\circ \bullet \circ \bullet \\[-1.5ex]
\circ \bullet \circ \bullet \\[-1.5ex]
\bullet \circ \bullet \circ \\[-1.5ex]
\circ \bullet \circ \bullet
\end{array}
\)
\\[3em]
\(\displaystyle
D\times V \ni (x,y) \mapsto
\begin{array}{r}
\bullet \circ \bullet \bullet \\[-1.5ex]
\circ \bullet \circ \circ \\[-1.5ex]
\bullet \circ \bullet \bullet \\[-1.5ex]
\circ \bullet \circ \circ
\end{array}
\)
& \quad \quad
\(\displaystyle
D\times D \ni (x,y) \mapsto
\begin{array}{r}
\bullet \circ \bullet \circ \\[-1.5ex]
\circ \bullet \circ \bullet \\[-1.5ex]
\bullet \circ \bullet \circ \\[-1.5ex]
\circ \bullet \circ \bullet
\end{array}
\)
\end{tabular}
\end{center}
\end{theorem}
\begin{proof}
Recall that
\[
\bigl(c(4n-3),c(4n-2),c(4n-1),c(4n)\bigr)
=
(1,0,1,c(n))
\]
for every \(n\in\mathbb{N}\).

In each \(4\times4\) block, the columns from left to right correspond to the second-coordinate values
\[
4y-3,\ 4y-2,\ 4y-1,\ 4y,
\]
whereas the rows from bottom to top correspond to the first-coordinate values
\[
4x-3,\ 4x-2,\ 4x-1,\ 4x.
\]
Therefore, the characteristic values along the horizontal and vertical directions are
\[
(1,0,1,c(y))
\qquad\text{and}\qquad
(1,0,1,c(x)),
\]
respectively.

By definition,
\[
T(u,v)=1
\quad\Longleftrightarrow\quad
c(u)\neq c(v).
\]
Hence, comparing the corresponding horizontal and vertical characteristic values produces one of the four displayed blocks, according to whether
\[
(c(x),c(y))
\]
is equal to
\[
(1,1),\qquad(1,0),\qquad(0,1),\qquad(0,0).
\]
These four cases correspond respectively to
\[
V\times V,\qquad
V\times D,\qquad
D\times V,\qquad
D\times D.
\]

Since the sets
\[
\{4n-3,4n-2,4n-1,4n\},
\qquad n\in\mathbb{N},
\]
partition \(\mathbb{N}\), the resulting \(4\times4\) blocks partition the entire tiling. This proves the stated quaternary substitution rule.
\end{proof}

The following corollary is a reframing of Theorem \ref{subthm}.

\begin{corollary}
\label{cor:factor-substitution}
The parameter tiling $T$ is a factor of a substitution tiling on the
four-state alphabet
\[
\mathcal{A}=\{(1,1), (1,0), (0,1), (0,0)\}.
\]
Let
\[
\sigma(1)=1011,
\qquad
\sigma(0)=1010,
\]
and define
\[
\Sigma(a,b)_{ij}
=
\bigl(\sigma(a)_i,\sigma(b)_j\bigr),
\qquad
(a,b)\in\mathcal{A},\quad 1\leq i,j\leq4.
\]
Thus $\Sigma$ is a constant-length $4\times4$ block substitution on
$\mathcal{A}$. Define the coding
\[
\psi\colon\mathcal{A}\to\{0,1\}
\]
by
\[
\psi(1,1)=\psi(0,0)=0,
\qquad
\psi(1,0)=\psi(0,1)=1.
\]
Then $T$ is obtained by applying $\psi$ entrywise to the substitution
tiling generated by $\Sigma$.
\end{corollary}

\section{Final remarks} \label{final}

\subsection{Motivation from parameter sets and outlook}

Our work was originally motivated by a recurring phenomenon in the study
of parametrised dynamical systems; namely that a classification of the individual
systems often reveals a further layer of structure in parameter space.
Perhaps the most familiar examples are the Mandelbrot set and the Rauzy
gasket. The Mandelbrot set arises by classifying complex
quadratic polynomials according to the connectedness of their Julia
sets, while the Rauzy gasket arises from a classification of dynamical
behaviour for a family of interval exchange transformations. In both
cases, the resulting parameter sets exhibit rich fractal structure,
with their recursive geometry reflecting renormalisation phenomena in
the underlying dynamical systems.

Our parameter tiling exhibits an analogous phenomenon in a discrete
setting. The aperiodic ordered geometry of $T$ is analogous to the unexpected ``ordered complexity'' (fractal nature) of the Mandelbrot set and the Rauzy gasket. Similarly, the substitution structure of $T$ is a discrete analogue of the recursive organisation produced by
renormalisation in the classical parameter spaces above. 

\begin{figure}[ht]
    \centering
    \begin{subfigure}{0.48\textwidth}
        \centering
        \includegraphics[width=\textwidth]{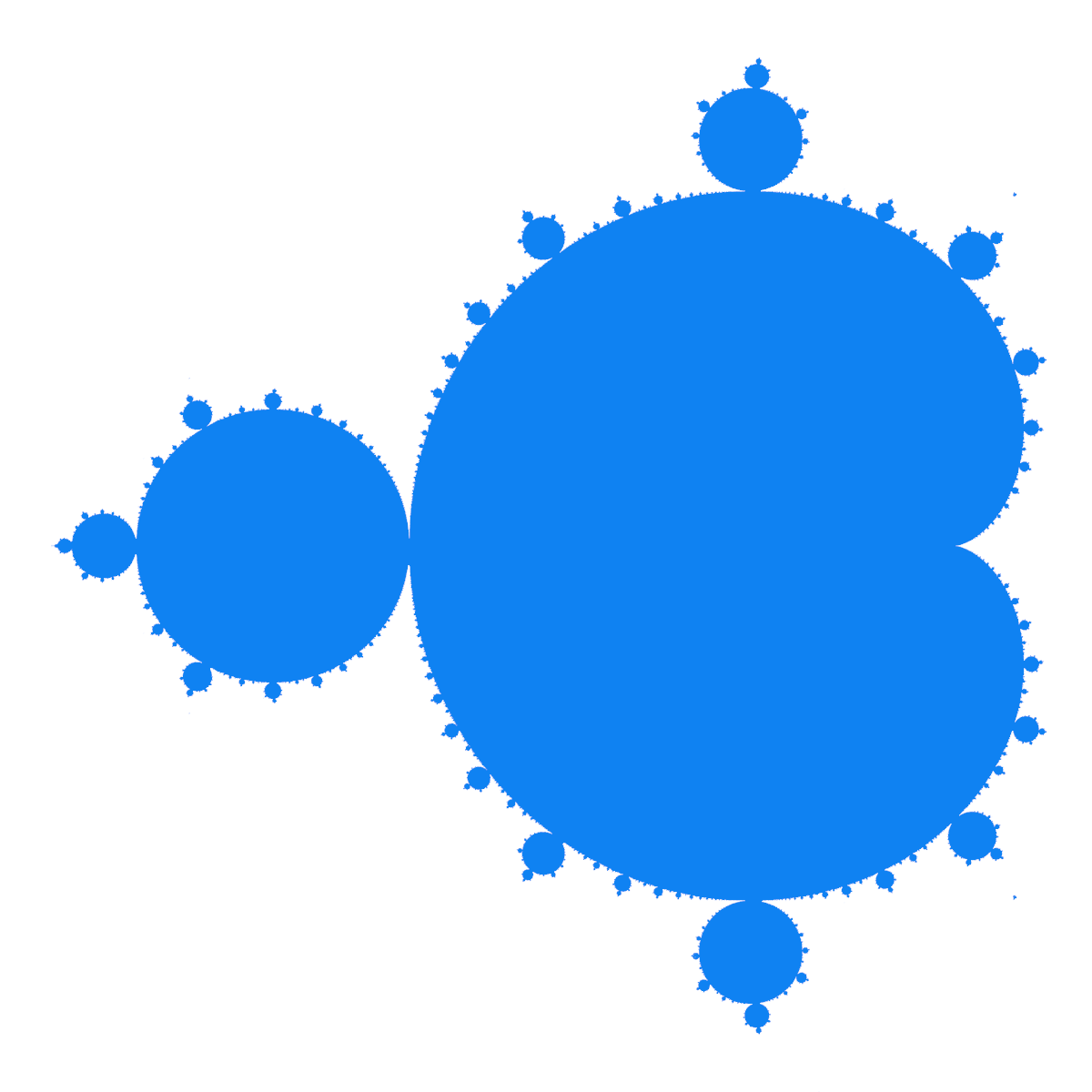}
    \end{subfigure}
    \hfill
    \begin{subfigure}{0.48\textwidth}
        \centering
        \includegraphics[width=\textwidth]{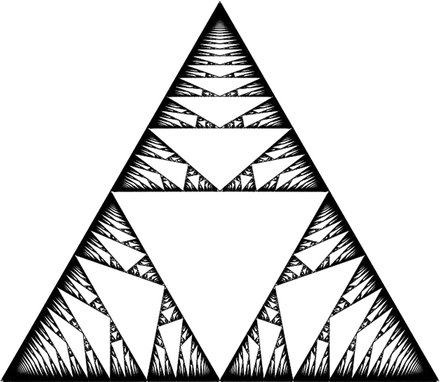}
    \end{subfigure}
    \caption{The Mandelbrot set and the Rauzy gasket are classical examples of fractals that emerge in parameter space through the classification of dynamical systems in a parametrised family.}
    \label{fig:two-images}
\end{figure}

The family studied in this paper is deliberately rigid, and it is
therefore natural to ask whether this phenomenon is specific to our
model or reflects a more general principle. In particular, one may ask:

\begin{quote}
\emph{For which natural classes of self-similar sets are the parameters
exhibiting dimension drop organised by an underlying dynamical or
combinatorial order?}
\end{quote}

More specifically, can such parameter sets exhibit substitution
structure? If so, can this be related systematically to the algebraic mechanisms responsible for dimension drop?

\subsection{The family with a variable contraction ratio}
It is worth noting that $q=1/4$ is the smallest positive contraction ratio for which the self-similar sets
$$
E_q(x,y)=\left\{ \sum_{n=1}^{\infty} \varepsilon_n q^n : (\varepsilon_n) \in \{0,x,y,x+y\}^{\mathbb{N}} \right\}
$$
of subsums of the corresponding bi-geometric series can exhibit all three possible topological types as $x,y\in\mathbb{N}$ vary. For $0<q<1/4$, the similarity dimension of the iterated function system defining $E_q(x,y)$ is less than $1$, implying that $E_q(x,y)$ is a Cantor set. It is also easy to show that if $1/2 \leq q < 1$, then
$$
E_q(x,y)=\left[0,\frac{q(x+y)}{1-q}\right].
$$
In the case $1/4<q<1/2$, the sets $E_q(x,y)$ can also have various topological types. However, their analysis is quite complex and remains largely unexplored in the general case.

\subsection{The case of rational parameters}

One can easily extend the results of Section~3 to the case of non-zero rational parameters $x$ and $y$. Indeed, let $x=p_1/q_1$ and $y=p_2/q_2$ for some $p_1,p_2\in\mathbb{Z}\setminus\{0\}$ and $q_1,q_2\in\mathbb{N}$. According to Lemma~\ref{T:Divisibility of one parameter}, $E(x,y)$ and $E_{\nu}(\eta x,\eta y)$ have the same topological type for arbitrary $\eta,\nu\in\mathbb{R}$ with $\eta\neq0$. Setting
$$
\eta=q_1q_2
\quad\text{and}\quad
\nu=-\min\left\{0,\,p_1q_2,\,p_2q_1,\,p_1q_2+p_2q_1\right\},
$$
we obtain
$$
E_{\nu}(\eta x,\eta y)
=
E\left(|p_1q_2|,|p_2q_1|\right).
$$
Thus, the problem is reduced to the case of natural parameters studied in the paper.

\subsection{The open set condition}

In the cases where $\mathcal{F}_{x,y}$ satisfies the OSC, the open set in the OSC can be chosen as $O=\operatorname{Int}E(x,y).
$
Indeed, each similarity $f_d(z)=(z+d)/4$, where $d\in\{0,x,y,x+y\}$, maps $O$ into itself, and the sets $f_d(O)$ are pairwise disjoint.
Thus, in Section 3.4, we found a necessary and sufficient condition under which a union of closed intervals is a homogeneous self-similar set that satisfies the open set condition. Problems of this kind were addressed in \cite{FHJ07}. For those sets $E(x,y)$ that are Cantorvals, the open set $\operatorname{Int}E(x,y)$ consists of countably many pairwise disjoint open intervals, whereas its boundary has a fractal structure. This demonstrates that even in the seemingly simple setting where the defining iterated function system satisfies the open set condition, its attractor may possess a rather intricate topological and fractal structure.

\subsection{Topological classification for self-similar sets}

As mentioned previously, there has been considerable work on the
topological classification of achievement sets of numerical series. A
fundamental result in this direction is the theorem of
\cite{NS}, which states that every achievement set of a convergent
positive series is either a finite union of intervals, a Cantor set, or a
Cantorval. Since not every self-similar set is an achievement set, this
raises the following natural question:

\begin{quote}
\emph{Can self-similar sets in the line be rigidly classified into a
finite collection of topological types?}
\end{quote}


\end{document}